\documentclass[1t2pt]{amsart}
\usepackage{amsfonts}
\usepackage[cp1250]{inputenc}
\usepackage{amssymb}
\usepackage{amsmath}
\usepackage{amsthm}
\usepackage{latexsym}
\usepackage{amsbsy}
\usepackage{graphicx}
\usepackage{color}
\usepackage{boxedminipage}%%%%%%%%%
\usepackage{geometry}
\usepackage[cp1250]{inputenc}
\usepackage{graphicx}
\usepackage{cite}%%%%%%%%%%

\newtheorem{theorem}{Theorem}
\newtheorem*{ttheorem}{Theorem}

\newtheorem*{theoremY}{Hayashi's Theorem}
\newtheorem*{theoremS}{Sarason's Theorem}
\newtheorem*{theoremS1}{ Theorem S1}
\newtheorem*{theoremS2}{ Theorem S2}

\newtheorem{proposition}{Proposition}
\newtheorem*{pproposition}{Proposition}
\newtheorem{lemma}{Lemma}
\newtheorem{corollary}{Corollary}
\theoremstyle{remark}
\newtheorem*{remark}{Remark}
\def\C{\mathbb{C}}
\def\D{\mathbb{D}}
\def\T{\mathbb{T}}

\def\H{\mathcal{H}}
\def\M{\mathcal{M}}%

\begin{document}

 \title{De Branges-Rovnyak spaces and local Dirichlet spaces of higher order}
 
 \author{Bartosz {\L}anucha, Ma{\l}gorzata Michalska, Maria Nowak, Andrzej So{\l}tysiak}

 \address{
 	Bartosz {\L}anucha  \newline Institute of Mathematics
 	\newline Maria Curie-Sk{\l}odowska University \newline pl. M.
 	Curie-Sk{\l}odowskiej 1 \newline 20-031 Lublin, Poland}
 \email{bartosz.lanucha@mail.umcs.pl}
 
 \address{
 	Ma{\l}gorzata Michalska  \newline Institute of Mathematics
 	\newline Maria Curie-Sk{\l}odowska University \newline pl. M.
 	Curie-Sk{\l}odowskiej 1 \newline 20-031 Lublin, Poland}
 \email{malgorzata.michalska@mail.umcs.pl}

\address{Maria Nowak
\newline Institute of Mathematics
\newline Maria Curie-Sk{\l}odowska University
\newline pl. M.
Curie-Sk{\l}odowskiej 1
\newline 20-031 Lublin, Poland}
\email{maria.nowak@mail.umcs.pl}

\address{Andrzej So{\l}tysiak
\newline Faculty of Mathematics and Computer Science
\newline Adam Mickiewicz University \newline ul. Uniwersytetu Pozna\'nskiego 4
\newline 61-614 Pozna\'n, Poland}
\email{asoltys@amu.edu.pl}



\thispagestyle{empty}
\begin{abstract} We discuss de Branges-Rovnyak spaces $\mathcal H(b)$ generated by nonextreme and rational functions $b$  and  local Dirichlet spaces of   order $m$  introduced  in \cite{LuoGuRich}. In  \cite{LuoGuRich} the authors   characterized nonextreme $b$ for which the  operator $Y=S|_{\mathcal H(b)}$, the restriction of the shift operator $S$ on $H^2$ to $\mathcal H(b)$, is  a strict $2m$-isometry and proved that such spaces $\mathcal H (b)$ are equal to local Dirichlet spaces of  order $m$. Here we give a characterization of local Dirichlet spaces of   order $m$  in terms of the $m$-th  derivatives that is a generalization of a known  result on  local Dirichlet spaces.
We also  find  explicit formulas for $b$  in the case when $\mathcal H(b)$ coincides with local Dirichlet space of   order $m$ with equality of norms. Finally, we prove  a property  of wandering vectors of $Y$   analogous to the property of  wandering vectors of the restriction of $S$ to  harmonically  weighted  Dirichlet  spaces obtained by D. Sarason in \cite{Sarason3}.
\end{abstract}
\maketitle

\noindent
\section{Introduction}
\vspace{.2in}

Let $\D$ denote the open unit disc in the complex plane $\C$, and let $\T=\partial\D$. For $\varphi\in L^{\infty}(\T)$ the Toeplitz operator on the Hardy space $H^2$ of the  disc $\D$ is defined by $T_{\varphi}f=P(\varphi f)$, where $P$ is the orthogonal projection of $L^2(\T)$ onto $H^2$. In particular, denote $S=T_{z}$.

For a function $b$ in the closed unit ball of $H^{\infty}$,
the \textit{de Branges--Rovnyak space $\mathcal{H}(b)$} is the image
of $H^2$ under the operator $(I-T_bT_{\overline{b}})^{1/2}$ with the corresponding range norm $\|\cdot\|_b$ (see the books \cite{Sarason1}, \cite{FM} and references given there).
It is known that $\mathcal{H}(b)$ is a Hilbert space with
reproducing kernel
  $$k_w^b(z)=\frac{1-\overline{b(w)}b(z)}{1-\overline{w}z}\quad(z,w\in\mathbb{D}).$$
In the case $b$ is an inner function,
$\mathcal H(b)= H^2 \ominus b H^2$
is the so called model space.

The theory of  $\H(b)$ spaces divides into two cases, according to whether $b$ is or is not an extreme point of the closed unit
ball of $H^{\infty}$. An important property of $\mathcal{H}(b)$ spaces is that they are $S^*$-invariant and in the case when $b$ is nonextreme they are also $S$-invariant. Moreover, for nonextreme $b$ the operator $Y=S_{|\mathcal{H}(b)}$ is bounded and it expands the norm.

Here we will  concentrate on the case when the function $b$ is not an extreme point
of the unit ball of $H^{\infty}$ (equivalently, $\log(1-|b|^2) $ is integrable on $\T)$.
Then there exists an outer
function $a\in H^{\infty}$ for which $|a|^2+|b|^2=1$ a.e. on
$\mathbb{T}$. Moreover, if we suppose  $a(0)>0$, then $a$ is
uniquely determined, and we say that $(b,a)$ is a pair.
If $(b,a)$ is a pair, then the quotient $\varphi=\frac{b}{a}$ is in the
Smirnov class $\mathcal{N}^{+}$.
Conversely,  every
nonzero function $\varphi\in\mathcal{N}^{+}$ has a unique representation (the so called canonical representation)
$\varphi=\frac{b}{a}$, where $(b,a)$ is a pair. In this case $T_{\varphi}$ is an unbounded operator on $H^2$ with the domain $\mathfrak
D(T_{\varphi})= \{f\in H^2: \ \varphi f\in H^2\}=aH^2$. Thus $T_{\varphi}$ is densely defined and closed. Consequently, its adjoint
$T^\ast_{\varphi}= T_{\overline{\varphi}}$ is densely defined and closed. Moreover,
$\mathfrak D(T_{\overline{\varphi}})= \mathcal H(b)$
and for $f\in\mathcal{H}(b)$,
  \begin{equation}
    \|f\|_{b}^2=\|f\|_{2}^2+\|T_{\overline{\varphi}}f\|_{2}^2.\label{norm1}
  \end{equation}
For more details on de Branges--Rovnyak spaces connection with unbounded Toeplitz operators see \cite{Sarason4}.

In papers \cite{Sarason2}, \cite{ransford}, \cite{ransford2} and \cite{LN} relations between de Branges-Rovnyak spaces and harmonically weighted Dirichlet spaces have been studied.

For a finite measure $\mu$ on $\T$ let $P_\mu$ denote the Poisson integral of $\mu$ given by
  $$P_{\mu}(z)=  \int_{\mathbb T}\frac {1-|z|^2}{|\zeta-z |^2}\,d\mu(\zeta), \quad z\in \D.$$
The associated \textit{harmonically weighted Dirichlet space} $\mathcal D_{\mu}$ consists of functions $f$ analytic in $\D$ for which
  \begin{equation}
    \mathcal D_{\mu}(f)= \int_{\mathbb D}|f'(z)|^2P_\mu(z)\,dA(z)<\infty\label{Dirichlet},
  \end{equation}
where $A$ denotes the normalized area measure on $\D$.

Spaces $\mathcal D(\mu)$ were introduced by S. Richter in \cite{Richter},  where it was proved that certain two-isometries on a Hilbert space can be represented as multiplication by $z$ on a space $\mathcal D(\mu)$.
The results on $\mathcal D(\mu)$ stated below were also obtained in \cite{Richter}.

If $\mu$ is a finite measure on $\T$ such that $\mu(\T)>0$, then $\mathcal D(\mu) \subset H^2$ and $\mathcal D(\mu) $ is a Hilbert space with the norm $\|\cdot\|_{\mathcal D(\mu)}$ given by
  \begin{equation*}
  \label{norm2}
    \|f\|_{\mathcal D(\mu)}^2=\|f\|_2^2+ \mathcal D_{\mu}(f).
  \end{equation*}
If $\mu=\delta_{\lambda}$ is the Dirac measure at $\lambda\in\mathbb T$, then
  $$\mathcal D_{\delta_\lambda}(f)=\mathcal D_{\lambda}(f)= \int_{\mathbb D}|f'(z)|^2\frac{1-|z|^2}{|z-\lambda|^2}\,dA(z)$$
and $\mathcal D_{\lambda}(f)$ is called the \textit{local Dirichlet integral} of $f$ at $\lambda$.

For $f\in H^2$, by Fubini's theorem, $\mathcal D_{\mu}(f)$ given by (\ref{Dirichlet}) can be expressed as
  $$\mathcal D_{\mu}(f)= \int_{\mathbb T}\mathcal D_{\lambda} (f)\,d\mu(\lambda).$$
Moreover, if $\lambda\in\T$ is such that $\mathcal D_{\lambda} (f)<\infty$, then the nontangential limit $f(\lambda)$ exists.

The following local Douglas formula for $\mathcal D_{\lambda} (f)$ was proved by Richter and Sundberg \cite{RichSun}: if $f\in H^2$ and $f(\lambda)$ exists, then
  \begin{equation*}
    \mathcal  D_{\lambda}(f)=\frac1{2\pi}\int_{\mathbb T}\left|\frac{f(z)-f(\lambda)}{z-\lambda}\right|^2|dz|.\label{norm3}
  \end{equation*}
If $f(\lambda)$ does not exist, then we set $\mathcal D_{\lambda}(f)=\infty$.
The  space   $\mathcal D(\delta_\lambda)=\mathcal D_{\lambda}$ is called the \textit{local Dirichlet space at} $\lambda$. It has also  been  proved in \cite{RichSun} that
  $$\mathcal D_{\lambda}= \{f\in\text{Hol}(\mathbb D): f= c+ (z-\lambda)g: c\in\mathbb C, \  g\in H^2\}.$$

In \cite{LuoGuRich} S. Luo, C. Gu and  S. Richter characterized nonextreme $b$ for which $Y$ is a strict $2m$-isometry, $m\in\mathbb N$. 
% (see \cite{LuoGuRich} for definition).
It turns out that the corresponding spaces $\mathcal{H}(b)$ are equal to the so called \textit{local Dirichlet spaces of order $m$}, $\mathcal D^m_{\lambda}$, $\lambda\in\mathbb T$, defined as follows:
  \begin{equation}
  \label{star}
	\mathcal D^m_{\lambda} =\{f\in\text{Hol}(\mathbb D): f= p+(z-\lambda)^mg, \ p \ \text{is a polynomial of degree} <m, \ g\in H^2\}.
  \end{equation}
They also showed that if $f\in H^2$, then $f\in \mathcal D_{\lambda}^m$ if and only if for each $j=0,1,\dots,m-1 $ the function $f^{(j)}$ has a nontangential limit at $\lambda$ and
  $$\mathcal D^m_{\lambda} (f)=\frac 1{2\pi}\int_{\mathbb T}\left|\frac{f(z)-T_{m-1}(f,\lambda)(z)}{(z-\lambda)^m}\right|^2|dz|<\infty,$$
where
  $$T_{m-1}(f,\lambda)(z)=f(\lambda)+f'(\lambda)(z-\lambda)+\dots+\frac{f^{(m-1)}(\lambda)}{(m-1)!}(z-\lambda)^{m-1}.$$
In the space $ \mathcal D^m_{\lambda}$ the norm is given by
  \begin{equation}
  \label{Dm}
	\|f\|^2_{\mathcal D_{\lambda}^m}=\|f\|^2_2+  \mathcal D_{{\lambda}}^m(f).
  \end{equation}
In the next section  we obtain the following characterization of the space  $\mathcal D^m_{\lambda}$.

\begin{theorem}
A function $f\in H^2$ is in the space $ \mathcal D^m_{\lambda}$ if and only if
  $$\int_{\mathbb D}\left|\frac{f^{(m)}(z)}{(z-\lambda)^m} \right|^2(1-|z|^2)^{2m-1} dA(z)<\infty.$$
\end{theorem}
\vspace{.2in}

In \cite{Sarason2} D. Sarason proved that if for $\lambda\in\mathbb{T}$,

  \begin{equation}
  \label{eqnorm}
    b_{\lambda}(z)=\frac{(1-\alpha)\overline{\lambda} z}{1-\alpha\overline{\lambda} z},\quad 
    \alpha =\frac{3-\sqrt{5}}{2}, \quad 
    \left( a_{\lambda}(z)=\frac{(1-\alpha)(1-\overline{\lambda} z)} {1-\alpha\overline{\lambda} z}\right),
  \end{equation}
then the space $\mathcal{D}_\lambda$ coincides with
$\mathcal{H}(b_\lambda)$ with equality of norms. It follows from the paper \cite{ransford} that $\H(b_\lambda)=\mathcal{D}_\lambda$ with  equality of norms only for $b_\lambda$ given by \eqref{eqnorm}.
\vspace{.2in}

In \cite{LuoGuRich} the authors also  described a relation between local Dirichlet space of higher order  $\mathcal D^m_\lambda$ and $\H(b)$. In particular, they obtained a necessary and sufficient condition for $\H(b)=\mathcal D^m_\lambda $ with equality of norms.
In Section 3, we derive explicit formulas for such functions $b$ analogous to \eqref{eqnorm}.

%In Section 4 we discuss some properties of $\mathcal{H}(b)$ generated by a nonextreme rational $b$.

%It is known that the backward shift operator $T_{\bar z}$ preserves the space $\mathcal{H}(b)$ and $T_{\bar z|\mathcal{H}(b)}$ is a contraction on $\mathcal{H}(b)$

A nonzero vector in a Hilbert space is called a wandering vector of a given operator if it is orthogonal to its orbit under the positive powers of the operator. In \cite{Sarason3} the author described the wandering vector of the shift operator $S_{\mu}=S_{|\mathcal D(\mu)}$ on the harmonically weighted Dirichlet space $ \mathcal D(\mu)$  associated with  a finitely atomic measure  $\mu=\sum_{i=1}^n\mu_i\delta_{\lambda_i}$ where $\lambda_1,\dots,\lambda_n$, are distinct points on $\T$ and $\mu_1,\dots,\mu_n$ are positive numbers.
One of his results states that the outer part of the wandering vector of  $S_{\mu}$ lies in the model space generated by a certain finite Blaschke product.
In the last section we consider the spaces $\H(b)$ when $b$ is  nonextreme  and rational,  and show that in this case the outer part of a wandering vector of
the operator $Y=S_{|\H(b)}$  has similar property.

\section{Local Dirichlet space of order $\mathbf m$}

In the proof of  Theorem 1 we will need the following technical lemma.

\begin{lemma}
\label{lem_aux_Dir_2m}
For a positive integer $m$ and $z,w\in\mathbb{D}$,
\begin{align}
\label{eq_aux_Dir_2m}
  \frac{\partial^{2m}}{\partial \overline{w}^m\partial z^m}\left(\frac{(\overline{w}-1)^m(z-1)^m}{1-\overline{w}z}\right)
  =(2m)!\frac{(\overline{w}-1)^m(z-1)^m}{(1-\overline{w}z)^{2m+1}}.
\end{align}
\end{lemma}
\begin{proof}
We proceed by induction. For $m=1$ the equality is obvious.

Assume \eqref{eq_aux_Dir_2m} holds true for an $m$. We will use the following Leibniz formula
\begin{align*}
  \frac{\partial^{2m}}{\partial \overline{w}^m\partial z^m} (f(\overline{w},z)\cdot g(\overline{w},z))
  =\sum_{\substack{0\leq k\leq m\\ 0\leq l\leq m}}\binom{m}{k}\binom{m}{l} \frac{\partial^{2m-k-l}}{\partial \overline{w}^{m-k}\partial z^{m-l}} f(\overline{w},z) \frac{\partial^{k+l}}{\partial \overline{w}^k\partial z^l} g(\overline{w},z).
\end{align*}
Hence
\begin{align*}
  &\frac{\partial^{2m+2}}{\partial \overline{w}^{m+1}\partial z^{m+1}}\left(\frac{(\overline{w}-1)^{m+1}(z-1)^{m+1}}{1-\overline{w}z}\right)\\
  &=\frac{\partial^{2}}{\partial \overline{w}\partial z}\left[\frac{\partial^{2m}}{\partial \overline{w}^{m}\partial z^{m}}\left(\frac{(\overline{w}-1)^{m}(z-1)^{m}}{1-\overline{w}z} \cdot (\overline{w}-1)(z-1)\right)\right]\\
  &=\frac{\partial^{2}}{\partial \overline{w}\partial z}\left[\sum_{\substack{0\leq k\leq m\\ 0\leq l\leq m}}\binom{m}{k}\binom{m}{l} \frac{\partial^{2m-k-l}}{\partial \overline{w}^{m-k}\partial z^{m-l}} \left(\frac{(\overline{w}-1)^{m}(z-1)^{m}}{1-\overline{w}z}\right) \frac{\partial^{k+l}}{\partial \overline{w}^k\partial z^l} \left((\overline{w}-1)(z-1)\right)\right].
\end{align*}
Since for $k\geq 2$ the derivatives  $\frac{\partial^{k}}{\partial \overline{w}^k} (\overline{w}-1)=\frac{\partial^{k}}{\partial z^k} (z-1)$ vanish,
\begin{align*}
 \frac{\partial^{2m+2}}{\partial \overline{w}^{m+1}\partial z^{m+1}}&\left(\frac{(\overline{w}-1)^{m+1}(z-1)^{m+1}}{1-\overline{w}z}\right)\\
  &=\frac{\partial^{2}}{\partial \overline{w}\partial z}
  \left[(\overline{w}-1)(z-1)\frac{\partial^{2m}}{\partial \overline{w}^{m}\partial z^{m}} \left(\frac{(\overline{w}-1)^{m}(z-1)^{m}}{1-\overline{w}z}\right) \right.\\
  &+m(z-1)\frac{\partial^{2m-1}}{\partial \overline{w}^{m-1}\partial z^{m}} \left(\frac{(\overline{w}-1)^{m}(z-1)^{m}}{1-\overline{w}z}\right)\\
  &+m (\overline{w}-1)\frac{\partial^{2m-1}}{\partial \overline{w}^{m}\partial z^{m-1}} \left(\frac{(\overline{w}-1)^{m}(z-1)^{m}}{1-\overline{w}z}\right)\\
  &\left.+m^2\frac{\partial^{2m-2}}{\partial \overline{w}^{m-1}\partial z^{m-1}} \left(\frac{(\overline{w}-1)^{m}(z-1)^{m}}{1-\overline{w}z}\right) \right].
\end{align*}
By the induction hypothesis,
\begin{align*}
  \frac{\partial^{2m+2}}{\partial \overline{w}^{m+1}\partial z^{m+1}} &\left(\frac{(\overline{w}-1)^{m+1}(z-1)^{m+1}}{1-\overline{w}z}\right)\\
 % &=(2m)!\left\{\frac{\partial^{2}}{\partial \overline{w}\partial z} \left[(\overline{w}-1)(z-1) \frac{(\overline{w}-1)^m(z-1)^m}{(1-\overline{w}z)^{2m+1}} \right] +m^2\frac{(\overline{w}-1)^m(z-1)^m}{(1-\overline{w}z)^{2m+1}}\right.\\
  % &\left.+\frac{\partial}{\partial z} \left[m(z-1) \frac{(\overline{w}-1)^m(z-1)^m}{(1-\overline{w}z)^{2m+1}} \right] +\frac{\partial}{\partial \overline{w}} \left[m (\overline{w}-1)\frac{(\overline{w}-1)^m(z-1)^m}{(1-\overline{w}z)^{2m+1}} \right]\right\}\\
  &=(2m)!\left\{\frac{\partial^{2}}{\partial \overline{w}\partial z} \left(\frac{(\overline{w}-1)^{m+1}(z-1)^{m+1}}{(1-\overline{w}z)^{2m+1}} \right) +m^2\frac{(\overline{w}-1)^m(z-1)^m}{(1-\overline{w}z)^{2m+1}}\right.\\
  &\left.+m\frac{\partial}{\partial z} \left(\frac{(\overline{w}-1)^m(z-1)^{m+1}}{(1-\overline{w}z)^{2m+1}} \right) +m\frac{\partial}{\partial \overline{w}} \left(\frac{(\overline{w}-1)^{m+1}(z-1)^m}{(1-\overline{w}z)^{2m+1}} \right)\right\}.
\end{align*}
A calculation shows that
  \begin{align*}
    \frac{\partial^{2}}{\partial \overline{w}\partial z} \left(\frac{(\overline{w}-1)^{m+1}(z-1)^{m+1}}{(1-\overline{w}z)^{2m+1}} \right)
    &=\frac{(\overline{w}-1)^{m}(z-1)^{m}}{(1-\overline{w}z)^{2m+3}}
    \left[2+4m+m^2-\overline{w}(2+5m+2m^2)\right.\\
    &\left.-z(2+5m+2m^2)+\overline{w}z(2+8m+6m^2) -\overline{w}^2z(m+2m^2)\right.\\
    &\left.-\overline{w}z^2(m+2m^2)+m^2\overline{w}^2z^2 \right],\\
    \\
    m\frac{\partial}{\partial z} \left(\frac{(\overline{w}-1)^m(z-1)^{m+1}}{(1-\overline{w}z)^{2m+1}} \right)
   % &=\frac{(\overline{w}-1)^{m}(z-1)^{m}}{(1-\overline{w}z)^{2m+2}}
   % \left[m+m^2-\overline{w}(m+2m^2)+m^2\overline{w}z\right]\\
    &=\frac{(\overline{w}-1)^{m}(z-1)^{m}}{(1-\overline{w}z)^{2m+3}}
    \left[m+m^2-\overline{w}(m+2m^2)-m\overline{w}z\right.\\
    &\left.+\overline{w}^2z(m+2m^2)-m^2\overline{w}^2z^2\right],\\
    \\
    m\frac{\partial}{\partial \overline{w}} \left(\frac{(\overline{w}-1)^{m+1}(z-1)^m}{(1-\overline{w}z)^{2m+1}} \right)
    % &=\frac{(\overline{w}-1)^{m}(z-1)^{m}}{(1-\overline{w}z)^{2m+2}}
   % \left[m+m^2-z(m+2m^2)+m^2\overline{w}z\right]\\
    &=\frac{(\overline{w}-1)^{m}(z-1)^{m}}{(1-\overline{w}z)^{2m+3}}
    \left[m+m^2-z(m+2m^2)-m\overline{w}z\right.\\
    &\left.+\overline{w}z^2(m+2m^2)-m^2\overline{w}^2z^2\right]
      \end{align*}
and%\text{and}& \\
\begin{align*}
    m^2\frac{(\overline{w}-1)^m(z-1)^m}{(1-\overline{w}z)^{2m+1}}
    &=\frac{(\overline{w}-1)^{m}(z-1)^{m}}{(1-\overline{w}z)^{2m+3}}
    \left[m^2-2m^2\overline{w}z+m^2\overline{w}^2z^2\right].
  \end{align*}
Thus
\begin{align*}
  \frac{\partial^{2m+2}}{\partial \overline{w}^{m+1}\partial z^{m+1}} &\left(\frac{(\overline{w}-1)^{m+1}(z-1)^{m+1}}{1-\overline{w}z}\right)\\
  &=(2m)!(2+6m+4m^2)\frac{(\overline{w}-1)^{m}(z-1)^{m}} {(1-\overline{w}z)^{2m+3}}\left[1-\overline{w}-z +\overline{w}z\right]\\
  &=(2m+2)!\frac{(\overline{w}-1)^{m+1}(z-1)^{m+1}} {(1-\overline{w}z)^{2m+3}}.
\end{align*}
\end{proof}

\begin{proof} [Proof of Theorem 1]
We will use Sarason's idea applied in the proof of Proposition 1 in \cite{Sarason2}. Without loss of generality we may assume that $\lambda=1$.
Let $A_m= T_{(z-1)^m}$  be the Toeplitz operator with the symbol $(z-1)^m$ on $H^2$. Let $\mathcal M(A_m)$ be the range of $A_m$ equipped with the Hilbert structure that makes $A_m$ a coisometry from $H^2$ onto $\mathcal M(A_m)$.  Since the kernel function for the space $H^2$ is $(1-\bar wz)^{-1}$, for $g\in  \mathcal M(A_m)$ we get
$$
g(w)=\left\langle g,\frac {(z-1)^m(\bar w -1)^m}{1-\bar w z}\right\rangle_{\mathcal M(A_m)},\quad w\in\mathbb D,
$$ where the last inner product is in the space $\mathcal M(A_m)$.
This implies that
\begin{align*}
  g^{(m)}(w)&=\left\langle g,\frac{\partial^m}{\partial \bar w^m}\frac {(z-1)^m(\bar w -1)^m}{1-\bar w z}\right\rangle_{\mathcal M(A_m)}\\[4pt]
  &=\left\langle  g^{(m)},\frac{\partial^{2m}}{\partial z^m\partial \bar w^m}\frac {(z-1)^m(\bar w -1)^m}{1-\bar w z}\right\rangle_{\mathcal M^{(m)}(A_m)},
\end{align*}
where the last inner product is taken in the range space $\mathcal M^{(m)}(A_m)$ of the operator of differentiation of order $m$.
By Lemma \ref{lem_aux_Dir_2m},
  $$ \frac{\partial^{2m}}{\partial z^m\partial \bar w^m}\frac {(z-1)^m(\bar w -1)^m}{1-\bar w z}= (2m)!\frac {(z-1)^m(\bar w -1)^m}{(1-\bar w z)^{2m+1}}.$$
Observe that by \eqref{star},
  $$ \mathcal M^{(m)}(A_m)=\{f^{(m)}:f\in\mathcal D^m_1\}.$$

We now note that the reproducing kernel for the space
  $$ A^2(\rho^m)=\left\{ h\in\text {Hol}(\mathbb D): \int_{\mathbb D}|h(z)|^2(1-|z|^2)^{2m-1}dA(z) <\infty\right\} $$
equals
  $$\frac{2m}{(1-\bar w z)^{2m+1}}.$$
Thus the space
$\{(z-1)^mh: h\in A^2(\rho^m)\}$ has the same (up to a constant) kernel as the space $\{f^{(m)}: f\in \mathcal D^m_1\}$. This means that $f\in\mathcal D^m_1$ if and only if $h=f^{(m)}/(z-1)^m\in A^2(\rho^m)$.
\end{proof}

\begin{corollary}
If a function $f$ holomorphic in $\D$ is such that for a $\lambda\in\mathbb T$,
  \begin{equation}
  \label{m}
  	\int_{\mathbb{D}}\left|\frac{f^{(m)}(z)}{(z-\lambda)^m}\right|^2(1-|z|^2)^{2m-1} dA(z)<\infty,
  \end{equation}
then for each $j=0,1,\dots,m-1 $ the function $f^{(j)}$ has a nontangential limit at $\lambda$.
\end{corollary}

\begin{proof} 
We can clearly assume that $\lambda=1$. It follows from the proof of Theorem 1 that  operator $T^m$ given by
  $$ T^m(g)=\frac 1{(2m-1)!}\frac {((z-1)^m g)^{(m)}}{(z-1)^m}$$ 
is an isometry  of $H^2$ onto $ A^2(\rho^m)$.  If $f$ satisfies condition \eqref{m}, then there exists $g\in H^2$ such that $f^{(m)}= ((z-1)^mg)^{(m)}$.
This means that $f(z) =(z-1)^mg+p$, where $p$ is a polynomial of degree $<m$, and the reasoning used in the proof of Lemma 9.1 in \cite{LuoGuRich} proves the claim.
\end{proof}

\section{De Branges-Rovnyak spaces $\mathcal{H}(b)$ and local Dirichlet spaces of finite order}

\vspace{.2in}

We first cite one of the main results contained in \cite{LuoGuRich} using notation from the Introduction. Recall that a Hilbert space operator $T$ is an $m$-isometry if
  $$\sum_{k=0}^m(-1)^{m-k}\binom{m}{k}T^{*k}T^k=0.$$
It is easy to check that every $m$-isometry is a $k$-isometry for every $k\geq m$. An $m$-isometry that is not an $(m-1)$-isometry is called a strict $m$-isometry.

\begin{ttheorem}[\!\!\cite{LuoGuRich}]
Let $b$ be a non--extreme point of the unit ball of $H^\infty$ with $b(0)=0$, and let $m\in\mathbb{N}$. Then $Y$ is not a strict $(2m+1)$-isometry, and the following are equivalent
\begin{enumerate}
  \item[{(i)}] $Y$ is a strict $2m$-isometry,
  \item[{(ii)}] $(b,a)$ is a rational pair such that $a$ has a single zero of multiplicity $m$ at a point $\lambda\in\T$,
  \item[{(iii)}] there is a $\lambda\in\T$ and a polynomial $\widetilde{p}$ of degree $<m$ with $\widetilde{p}(\lambda)\ne0$ such that
  \begin{equation*}
  \label{orderm}
	\Vert f\Vert^2_b=\Vert f\Vert^2_2+\mathcal{D}^m_\lambda(\widetilde{p}f).
  \end{equation*}
\end{enumerate}
If the three conditions hold, then there are polynomials $p$ and $r$ of degree $\leq m$ such that $b=\frac{p}{r}$, $a=\frac{(z-\lambda)^m}{r}$, and $\widetilde{p}(z)=z^m\overline{p\left(\frac{1}{\bar{z}}\right)}$ for $z\in\mathbb{D}$, and  $|r(z)|^2=|p(z)|^2+|z-\lambda|^{2m}$ for all $z\in\T$. Furthermore, $\H(b)=\mathcal{D}_\lambda^m$ with equivalence of norms.
\end{ttheorem}

In particular,  $\H(b)=\mathcal{D}_\lambda^m$ with equality of norms if and only if $b(z)=\frac{z^m}{r(z)}$, where $r$ is a polynomial of degree $m$ that has no zeros in $\overline{\D}$ and such that $|r(z)|^2=1+|z-\lambda|^{2m}$ for all $z\in\T$.

Now  we give another proof of the sufficient condition for  $\H(b)=\mathcal{D}^m_\lambda$ with equality of norms.% using the theory of $\H(b)$ spaces.

The following proposition is actually contained in the above mentioned result of \cite{LuoGuRich}.

\begin{pproposition}[\!\!\cite{LuoGuRich}]
If for $\lambda \in\mathbb T$  the function $\varphi_{\lambda}^m\in \mathcal{N}^{+}$ is defined by
  \begin{equation}
  \label{LGM}
	\varphi_{\lambda}^m (z) =\frac{(\bar\lambda z)^m}{(1-\bar \lambda z)^m},\quad z\in\mathbb D,
  \end{equation}
and its canonical representation is $\varphi_{\lambda}^m =\frac {b_{\lambda}}{a_{\lambda}}$, then $\mathcal H(b_{\lambda})= \mathcal D_{\lambda}^m$ with equality of norms.
\end{pproposition}

\begin{proof} %\textbf{1}.
We first show that for $f\in \text{Hol}(\overline{\mathbb{D}})$ -- the space of functions holomorphic on $\overline{\mathbb D}$,
  \begin{equation}
  \label{T}
    T_{\overline{\varphi_{\lambda}^m}}f(z)=\lambda^m \frac{f(z)-T_{m-1}(f,\lambda)(z)}{(z-\lambda)^m}
   \end{equation}
Let for $\lambda\in\mathbb T$,
  $$k_{\lambda}(z)=\frac 1{1-\overline{\lambda} z},\quad z\in \mathbb D.$$

It has been derived in \cite{LN} that for $f\in\text{ Hol}(\overline{\mathbb{D}})$,
  $$T_{\overline{k}_{\lambda}}f(z)=f(z)+\lambda\frac{f(\lambda)-f(z)}{\lambda-z}.$$
Since
  $$\varphi_{\lambda}^1(z)= \frac{\overline{\lambda}z}{1-\overline{\lambda} z},$$
we get
  $$T_{\overline{\varphi_{\lambda}^1}} f(z)= T_{\overline {k_{\lambda}-1}}f(z)= \lambda\frac{f(z)-f(\lambda)}{z-\lambda}$$
which proves (\ref{T}) for $m=1$.
Assume now that (\ref{T})  holds for $m-1$. Since for  a nonextreme $b$  the space $\text{Hol}(\overline{\mathbb{D}})$ is a subspace of  $\mathcal H(b)$ and $T_{\overline{\varphi}}(\text{Hol}(\overline{\mathbb{D}}))\subset \text{Hol}(\overline{\mathbb{D}})$, $\varphi=\frac{b}{a}$
(see  \cite[Lemma 6.1]{Sarason4}), we get
  \begin{align*} 
    T_{\overline{\varphi_{\lambda}^m}} f(z)
    &=T_{\overline{\varphi_{\lambda}^1}} T_{\overline{\varphi_{\lambda}^{m-1}}} f(z) \\
    &=T_{\overline{\varphi_{\lambda}^1}}\left(\lambda^{m-1}\frac{f(z)- T_{m-2}(f,\lambda)(z)}{(z-\lambda)^{m-1}}\right)
    = \lambda^m \frac{f(z)-T_{m-1}(f,\lambda)(z)}{(z-\lambda)^m},
  \end{align*}
where the last equality follows from the fact that for $f\in\text{Hol}(\overline{\mathbb{D}})$,
  $$\lim_{z\to\lambda}\frac{f(z)-T_{m-2}(f,\lambda)(z)}{(z-\lambda)^{m-1}}=\frac{f^{(m-1)}(\lambda)}{(m-1)!}.$$

Note now that polynomials are a dense subset of both $\mathcal H(b_{\lambda})$ and $\mathcal{D}_{\lambda}^m$, and by \eqref{norm1}, \eqref{Dm} and \eqref{T}, for each polynomial $f$,   
  $$\|f\|_{b_{\lambda}}=\|f\|_{\mathcal{D}_{\lambda}^m}.$$
Since the norms $\|\cdot\|_{b_{\lambda}}$ and $\|\cdot\|_{\mathcal{D}_{\lambda}^m}$ are equivalent, we get $\mathcal H(b_{\lambda})=\mathcal{D}_{\lambda}^m$ with equality of norms.
\end{proof}

%\begin{remark}
%Now assume that $f\in \mathcal H(b_{\lambda})$. %, where the nonextreme $b$ satisfies the hypothesis.
%Since polynomials are dense in $\mathcal  H(b_{\lambda})$,  we can choose a sequence of polynomials $\{p_n\}$ such that $p_n\to f$ in $\mathcal H(b_{\lambda})$. It follows from Sarason's results \cite[(VII-2)]{Sarason1} (see also  \cite{NSS}) that for every $k=0,1,\dots, m-1$ and $f\in\mathcal H(b_{\lambda})$ the nontangenial limit $f^{(k)}(\lambda)$ exists and the functional $f\to f^{(k)}(\lambda)$ is bounded in $\mathcal H(b_{\lambda})$. This implies that for $z\in\mathbb D$,
%\begin{align*}
%T_{\overline{\varphi_{\lambda}^m}}f(z)&= \lim_{n\to\infty}T_{\overline{\varphi_{\lambda}^m}} p_n(z)= %\lim_{n\to\infty}
%\lambda^m \frac{p_n(z)-T_{m-1}(p_n,\lambda)(z)}{(z-\lambda)^m}\\[4pt]
%&= \lambda^m \frac{f(z)-T_{m-1}(f,\lambda)(z)}{(z-\lambda)^m}.
%\end{align*}
%\end{remark}

In the following  proposition we find the explicit formulas for $b_{\lambda}$ for which $\H(b_{\lambda}) =\mathcal D_{\lambda}^m$ with equality of norms.

\begin{proposition}
Let for $\lambda \in\mathbb T$ and $m\in\mathbb N$  the function $\varphi_{\lambda}^m\in \mathcal{N}^{+}$ be defined by \eqref{LGM}.
Then the canonical representation $(b_{\lambda}, a_{\lambda})$ of $\varphi_{\lambda}^m$ is given by
  $$b_{\lambda}(z)=\frac{C_m(\bar\lambda z)^m}{\prod_{k=1}^m(1 - \bar  \lambda z_k z)},\quad a_{\lambda}(z)=\frac{C_m(1-\bar\lambda z)^m}{\prod_{k=1}^m(1 - \bar  \lambda z_k z)}$$
with
$C_m= \prod_{k=1}^m(1-z_k)$, where $z_k$ are preimages of the $m$-th roots of $1$ for odd m and $m$-th roots of $-1$ for even $m$ under the Koebe function $k(z)=\frac{z}{(1-z)^2}$.
\end{proposition}

\begin{proof}
Assume that $\varphi_{\lambda}^m=\dfrac{b_\lambda}{a_\lambda}$ where $ \left(b_\lambda,a_\lambda\right)$ is a pair.
Then for $|z|=1$
  $$1+|1-\bar\lambda z|^{2m} =\frac{|1-\bar\lambda z|^{2m}} {|a_\lambda(z)|^2}.$$
By Fej\'{e}r-Riesz theorem there is a unique polynomial $r$ of degree $m$ without zeros in $\overline{\mathbb D}$, such that $r(0)>0$ and
\begin{equation}
	1 + |1-\bar\lambda z|^{2m}=|r(z)|^2.\label{rz}
\end{equation}
Now define a polynomial
  $$w(z)= z^m + (2 z -z^2-1)^m = z^m+(-1)^m(1-z)^{2m} $$ 
and observe that for $|z|=1$,
  $$|w(\bar\lambda z)|= 1 + |1-\bar\lambda z|^{2m}.$$
Next note that $w(z)=0$ if and only if
  $$\left(\frac{z}{(1-z)^2}\right)^m=(-1)^{m+1}.$$
This implies that, in the case when $m$ is  odd, the zeros of $w$ in $\D$ are $z_k=g(e_k),$ $k=0,1,\dots, m-1$ where $g $ is the inverse of the Koebe function $k(z)=\frac z{(1-z)^2}$   and $e_k$ are the $m$-th roots of $1$.
In the case $m$ is even, the zeros of $w$ in $\mathbb D$ are $z_k=g(\widetilde{e}_k),$ $k=1,\dots m$, where $\widetilde{e}_k$ are the  $m$-th roots of $-1$.
The other zeros of $w$ are $\frac 1{\bar z_k}$, $k=1,\dots, m$. Consequently, the zeros of $w(\bar \lambda z)$ in $\D$ are $\lambda z_k$, $k=1,\dots m$. Thus
  $$ w(\bar\lambda z)=c\prod_{k=1}^m(z-\lambda z_k)(z-\frac{1}{\overline{\lambda z_k}}).$$
Since for $|z|=1$
  $$|z-\lambda z_k|= |z\overline{\lambda z_k}-1|,$$
the polynomial $r(z)$ is given by
  $$r(z)= c_0\prod_{k=1}^m(1 - \bar  \lambda z_k z),$$
where in view of \eqref{rz}, $c_0= (\prod_{k=1}^m(1-z_k))^{-1}$.
We also  mention that
  $$\min_{|z|=1}|r(z)| = 1 =r(\lambda). $$
Consequently, we get
  $$b_\lambda(z)= \frac {(\bar\lambda z)^m}{r(z)},\quad  a_\lambda(z)= \frac{(1-\bar\lambda z)^m}{r(z)},$$
which proves the proposition.
\end{proof}

\vspace{.2in}
\section{ $\H(b)$ generated by nonextreme rational $b$}%result shows

Let $(b,a)$  be a rational pair and let $\lambda_1,\lambda_2,\dots\lambda_m $ be all the zeros of $a$ on $\T$, listed according to multiplicity.
It has  been proved in \cite{ransford2} that for every such rational pair
\begin{equation}\label{spp}
  \mathcal H(b)=\left\{\prod_{j=1}^m(z-\lambda_j)g +s:\ g\in H^2,\ s\in \mathcal{P}_{m-1}\right\},
\end{equation}
where $\mathcal{P}_{n}$ denotes the set of polynomials of degree at most $n$. This means that the space $\mathcal H(b)$  is in fact determined by the zeros of $a$ on $\mathbb{T}$. Moreover, in the case when $\lambda_1=\lambda_2=\ldots=\lambda_m=\lambda$, by \eqref{star}, $\mathcal H(b)=\mathcal D_\lambda^m$ (as sets).

For fixed $\lambda_1,\lambda_2,\dots\lambda_m \in\mathbb{T}$ and a polynomial $p$ such that $p(\lambda_j)\ne 0$, $j=1,2,\ldots,m$, consider the function $\varphi\in\mathcal{N}^{+}$ defined by
  \begin{equation}
  \label{fi}
	\varphi(z)=\dfrac{p(z)} {\prod_{j=1}^m\left(1-\overline{\lambda}_jz\right)}=\dfrac{b(z)}{a(z)},
	%=\dfrac{b_m(z)}{a_m(z)}
  \end{equation}
where $(b,a)$ is a canonical pair. It is worth noting that the pair $(b,a)$ is rational. Moreover, by the aforementioned result from \cite{ransford2}, for each polynomial $p$, the space $\mathcal H(b)$ is described by \eqref{spp} and the corresponding norms $\|\cdot\|_b$ given by \eqref{norm1} are equivalent.

For a positive integer $m$ set
  \begin{equation*}
  \label{1raz}
	\varphi_m(z)=\dfrac{z^m}{\prod_{j=1}^m \left(1-\overline{\lambda}_jz\right)}=\dfrac{b_m(z)}{a_m(z)}.
  \end{equation*}
%and
%\begin{equation}
%\varphi(z)=\frac{p(z)}{\prod_{j=1}^m\left(1-\overline{\lambda}_jz\right)}=\dfrac{\alpha_0 + \alpha_1z+\dots +\alpha_mz^m}{\prod_{j=1}^m\left(1-\overline{\lambda}_jz\right)}
%=\dfrac{b(z)}{a(z)},\label{fi}
%\end{equation}
%where $p$ is a polynomial of degree $\leq m$ such that $p(\lambda_j)\ne 0$.
%Then clearly  $\lambda_j$, $j=1,\dots , m$, are the only zeros on $\mathbb{T}$ of both $a_m$ and $a$. It follows that $\H(b_m)=\H(b)$ as sets and the norms $\|\cdot\|_b$ and $ \|\cdot\|_{b_m}$ are equivalent. Moreover, both $a_m/a$ and $a/a_m$ belong to $H^\infty$ \cite[Vol. 2, Cor. 27.17]{FM} and so $\mathfrak D(T_{{\varphi_m}})=a_mH^2=aH^2=\mathfrak D(T_{{\varphi}})$.

The next proposition slightly extends the result contained in \cite [Theorem 9.4]{LuoGuRich}.

\begin{proposition} 
Let $\varphi=b/a$ be defined by \eqref{fi} with the polynomial $p$ of degree $\leq m$, and let $q$ be a polynomial of degree $\leq m$ such that $q(\lambda_j)\ne 0$ for $j=1,\ldots,m$. Then
  \begin{equation}
  \label{norm}
	\|f\|^2_b=\|f\|^2_2+\|T_{\overline{\varphi}_m}(qf)\|^2_2\quad \text{for each }f\in\mathcal{H}(b)
  \end{equation}
if and only if, up to a multiplicative unimodular constant, $q(z) =\widetilde{p}(z)=z^m \overline{p\left(\frac 1{\bar z}\right)}$.
\end{proposition}

\begin{proof}
Assume that $q=\widetilde{p}$. Then, by Proposition 6.5 in \cite{Sarason4},
  $$T_{\overline{\varphi}}f
  =T_{\frac{1}{\overline{\prod_{j=1}^m(1-\overline{\lambda}_jz)}}} T_{\overline{p}}f
  = T_{\frac{1}{\overline{\prod_{j=1}^m(1-\overline{\lambda}_jz)}}}
  T_{\bar z^m}(\widetilde{p}f)
  =T_{\overline{\varphi}_m}(\widetilde{p}f).$$

Assume now that \eqref{norm} holds with a polynomial $q$. Since
  $$\|f\|^2_b=\|f\|^2_2+\|T_{\overline{\varphi}}f\|^2_2,$$
we get
  $$\|T_{\overline{\varphi}}f\|^2_2=\|T_{\overline{\varphi}_m}(qf)\|^2_2.$$
On the other hand,
  $$T_{\overline{\frac{\widetilde{q}} {\prod_{j=1}^m(1-\overline{\lambda}_jz)}}}f
  =T_{\frac{1}{\overline{\prod_{j=1}^m(1-\overline{\lambda}_jz)}}} T_{\overline{\widetilde{q}}}f
  =T_{\frac{1}{\overline{\prod_{j=1}^m(1-\overline{\lambda}_jz)}}} 
  T_{\bar z^m}(qf)=T_{\overline{\varphi}_m}(qf),$$
and so
  $$\|T_{\overline{\varphi}}f\|^2_2=\|T_{\overline{\varphi}_m}(qf)\|^2_2 =\|T_{\overline{\frac{\widetilde{q}} {\prod_{j=1}^m(1-\overline{\lambda}_jz)}}}f\|^2_2.$$
Consequently, if
  \begin{equation*}
  \label{fist}
	\varphi_{*}(z)=\frac{\widetilde{q}(z)}{\prod_{j=1}^m\left(1-\overline{\lambda}_jz\right)}
	=\dfrac{b_{*}(z)}{a_{*}(z)},
  \end{equation*}
then $\mathcal{H}(b)=\mathcal{H}(b_{*})$ with equality of norms. By \cite[Vol. 2, Cor. 27.12]{FM}, $b=c_{*}\cdot b_{*}$ for some unimodular constant $c_{*}$. Since $|a|=|a_{*}|$ on $\mathbb{T}$
and the modulus of an outer function on $\mathbb{T}$ determines that function up to a multiplicative unimodular constant,
we obtain that $q=c\cdot\widetilde{p}$ for some $c\in\mathbb{T}$.
\end{proof}
%
%\begin{proof}
%	We first show that equality \eqref{norm} holds if $q=\widetilde{p}$. To this end assume that
%	$f^+=T_{\overline \varphi}f$, or equivalently
%	$$
%	T_{\overline{p}}f= T_{\overline{\prod_{j=1}^m\left(1-\overline{\lambda}_jz\right)}}f^+,
%	$$
%	Thus
%	$$
%	T_{\overline{p}}f= P(\bar z^m(\overline{\alpha}_0z^m+\overline{\alpha}_1z^{m-1}+\ldots+\overline{\alpha}_{m})f)= T_{\bar z^m}(\widetilde{p}f)
%	$$
%	and, consequently,
%	$$
%	f^+=  T_{\frac{\bar z^m}{\overline{\prod_{j=1}^m\left(1-\overline{\lambda}_jz\right)}}}(\widetilde{p}f)=T_{\overline{\varphi}_m}(\widetilde{p}f).
%	$$
%	Assume now that that equality \eqref{norm} holds for $q$ and let $ \widetilde{f}= T_{\overline{\varphi}_m}(qf)$. Then
%	\begin{equation} T_{\overline{z}^m}(qf)=T_{ \prod_{j=1}^m\left(1-\lambda_j\overline z\right)}\widetilde{f}.\label{tilde}\end{equation}
%	Put $q(z)=\beta_0+ \beta_1z+\dots+\beta_m z^m$. Then $\widetilde{q}(z)=\overline{\beta}_0z^m+\overline{\beta}_1z^{m-1} +\dots+\overline{\beta}_m $ on $\mathbb{T}$ and by \eqref{tilde},
%	$$ T_{\overline{\widetilde {q}}} f= T_{\overline z^m}(qf)= T_{ \prod_{j=1}^m\left(1-\lambda_j\overline z\right)}\widetilde{f}.$$
%	If $\widetilde b/\widetilde a$ is the canonical representation of $ \widetilde {q}\left(\prod_{j=1}^m\left(1-\overline{\lambda_j}z \right)\right)^{-1}$,
%	then equality \eqref{norm}  means that $\|f\|_{b}=  \|f\|_{\widetilde b}$, which implies that $b=c\widetilde b$ with unimodular constant $c$.
%\end{proof}

We now prove the following
\begin{theorem}
\label{wandering}
Let $ (b,a)$ be a rational pair such that the corresponding Smirnov function $\varphi$  is given by \eqref{fi} with $p(z)=\alpha_0+\alpha_1z+ \dots + \alpha_mz^m$ such that $p(\lambda_j)\ne0,$ $j=1,2\dots m$.
Then there is a finite Blaschke product $B_m$ of degree $m$, such that the outer part of a wandering vector of $Y$  lies in $\H(zB_m)$.
%, where $B_m$ denotes the finite Blaschke product with zeros  $z_k= \frac1{\overline{w_k}}$, $k=1,2\dots,m$.
\end{theorem}

\begin{proof}
We now use reasoning from the proof of Fej\'{e}r-Riesz theorem (see, e.g., \cite{helton}). Set
  $$v(e^{it})= \left|\prod_{j=1}^m\left(1-\overline{\lambda}_je^{it}\right)\right|^2+ |p(e^{it})|^2=\sum_{j=-m_0}^{m_0}c_j e^{ijt},$$
where $m_0\leq m$ and $\bar c_{j}=c_{-j}$, $j=1,2,\ldots,m_0$. Then $v(e^{it})>0$ and
  $$w(z)=z^{m_0}v(z)=\sum_{j=0}^{2{m_0}}c_{j-{m_0}} z^{j}, \quad z\in\mathbb{D},$$
is an analytic polynomial of degree $2{m_0}$ with no zeros on $\mathbb{T}$ and such that $w(0)\neq0$. Moreover, the zeros of $w$ are $w_k$ and $\frac1{\overline{w}_k}$, where $w_k\in \mathbb C\setminus\overline{\mathbb D}$, $k=1,2,\dots,{m_0}$. Therefore
  \begin{equation}
  \label{WW}
	w(z)=c\prod_{j=1}^{m_0}(z-w_j)(1-\overline{w}_jz),\ c>0.
  \end{equation}
%and the desired $r$ is given by 
%$$r(z)=\sqrt{c}\prod_{j=1}^m(z-w_k).$$

With the notation above we define the function
  $$ W(z)=\prod_{j=1}^m\left(1-\overline{\lambda}_jz\right)^2 +(-1)^m\overline{\lambda} p(z)\widetilde{p}(z),$$
where $\lambda=\lambda_1\cdot\lambda_2\cdot\ldots\cdot\lambda_m$. Observe that for $|z|=1$,

  \begin{align*}
    W(z)&=(-1)^m\overline{\lambda}z^m
    \left(\prod_{j=1}^m\left(1-\overline{\lambda}_jz\right)
    \left(1-{\lambda}_j\overline{z}\right) + \overline{z}^m\widetilde{p}(z)p(z)\right)\\ 
    &=(-1)^m\overline{\lambda}z^m\left(\prod_{j=1}^m
    \left(1-\overline{\lambda}_jz\right)
    \left(1-{\lambda}_j\overline{z}\right) +\overline{p(z)}p(z)\right)\\
    &=(-1)^m\overline{\lambda}z^{m}v(z) =(-1)^m\overline{\lambda}z^{m-m_0}w(z).
    %(-1)^m\overline{\lambda}z^m|r(z)|^2,
  \end{align*}
%we see that $|W(z)|=|r(z)|^2.$
Now we apply the idea from the proof of Lemma 3 in \cite{Sarason3}.  If we assume that $h$ is a wandering vector of $Y$, then for $k=0,1,\dots$,
  \begin{align*}
    0=& \left\langle z^{k+1}\prod_{j=1}^m \left(1-\overline{\lambda}_jz\right)^2 h, h\right\rangle_b\\
	=&\left\langle z^{k+1}\prod_{j=1}^m \left(1-\overline{\lambda}_jz\right)^2 h, h\right\rangle_{2} +\left\langle T_{\overline{\varphi}} \left(z^{k+1}\prod_{j=1}^m \left(1-\overline{\lambda}_jz\right)^2 h\right), T_{\overline{\varphi}}\left(h\right)\right\rangle_{2}\\
%=&\left\langle z^{k+1}\prod_{j=1}^m\left(1-\overline{\lambda}_jz\right)^2 h, h\right\rangle_{2}+\left\langle P\left(\overline{\varphi} z^{k+1}(-1)^m\overline{\lambda}z^m\prod_{j=1}^m\left(1-\overline{\lambda}_jz\right)\left(1-{\lambda}_j\overline{z}\right) h\right), T_{\overline{\varphi}}\left(h\right)\right\rangle_{2}\\
	=&\left\langle z^{k+1}\prod_{j=1}^m \left(1-\overline{\lambda}_jz\right)^2 h, h\right\rangle_{2} +\left\langle P\left(\frac{\overline{p(z)}z^{k+1} \prod_{j=1}^m\left(1-\overline{\lambda}_jz\right)^2 h} {\prod_{j=1}^m\left(1-{\lambda}_j\overline{z}\right)}\right), T_{\overline{\varphi}}\left(h\right)\right\rangle_{2}\\
	=&\left\langle z^{k+1}\prod_{j=1}^m \left(1-\overline{\lambda}_jz\right)^2 h, h\right\rangle_{2} +\left\langle (-1)^m\overline{\lambda}z^m \overline{p(z)}z^{k+1} \prod_{j=1}^m\left(1-\overline{\lambda}_jz\right) h, T_{\overline{\varphi}}\left(h\right)\right\rangle_{2}\\
	=&\left\langle z^{k+1}\prod_{j=1}^m \left(1-\overline{\lambda}_jz\right)^2 h, h\right\rangle_{2} +\left\langle\varphi(z) (-1)^m\overline{\lambda} \widetilde{p}(z)z^{k+1}\prod_{j=1}^m \left(1-\overline{\lambda}_jz\right) h, h\right\rangle_{2}\\
	=&\left\langle z^{k+1}\prod_{j=1}^m \left(1-\overline{\lambda}_jz\right)^2 h, h\right\rangle_{2} +\left\langle (-1)^m\overline{\lambda}p(z)\widetilde{p}(z)z^{k+1}h, h\right\rangle_{2} =\left\langle z^{k+1}Wh_0,h_0\right\rangle_{2},\\
  \end{align*}
where $h_0$ denotes the outer part of $h$. It follows that $h_0$ is orthogonal to the subspace of $H^2$ spanned by $z^{k+1}Wh_0$, $k=1,2,\ldots$. Since $h_0$ is an outer function this subspace is $zB_m H^2$, where $B_m=z^{m-m_0}B_{m_0}$ and $B_{m_0}$ is a Blaschke product with zeros $z_k=\frac1{\overline{w}_k}\in \mathbb{D}$. Hence $h\in\H(zB_m)=H^2\ominus zB_m H^2$.
\end{proof}

\begin{remark}
Let $b$ be a nonextreme function determined by $\varphi=\frac{b}{a}$ from Theorem \ref{wandering}. We remark that every function of the form $ua$, where $u$ is inner, is a wandering vector of $Y$. This follows from the fact that multiplication by $a$ is an isometry of $H^2$ into $\mathcal{H}(b)$. We first note that by \eqref{WW} we have   
  $$a(z)=\frac{\prod_{j=1}^m\left(1-\overline{\lambda}_jz\right)}{r(z)}\quad\text{and}\quad b(z)=\frac{p(z)}{r(z)},$$ 
where $r(z)=\sqrt{c}\prod_{j=1}^m(z-w_k)$. Consequently, for $f\in H^2$,
  \begin{align*}
	\|af\|_b^2=
	&\|af\|_2^2+\|T_{\overline{\varphi}}(af)\|_2^ 2=\|af\|_2^2+\|P\left(\overline{\varphi}af\right)\|_2^2\\
	=&\|af\|_2^2+\left\|P\left(\frac{\overline{p(z)}} {\prod_{j=1}^m\left(1-{\lambda}_j\overline{z}\right)} \frac{\prod_{j=1}^m\left(1-\overline{\lambda}_jz\right)} {r(z)}f\right)\right\|_2^2=\|af\|_2^2+\left\|P\left(\frac{z^m\overline{p(z)}}{r(z)}f\right) \right\|_2^2 \\=&\|af\|_2^2+\left\|bf\right\|_2^2=\|f\|_2^2.
  \end{align*}
It follows that for an inner function $u$ and every $k\geq 0$,
  \begin{align*}
	\langle Y^{k+1}(ua),ua\rangle_{b} 
	=\langle ua z^{k+1},ua\rangle_{b}
	=\langle u z^{k+1},u\rangle_{2}
	=\langle z^{k+1},1\rangle_{2}=0.
  \end{align*}
\end{remark}

\end{document}